\documentclass[11pt]{article}
\usepackage{epic,latexsym,amssymb}
\usepackage{color}
\usepackage{tikz}
\usepackage{comment}
\usepackage{lineno}
\usepackage{pstricks,pst-node,pst-plot}
\usetikzlibrary{patterns}
\usepackage{geometry,graphicx}

\usepackage{tikz,multicol,bm}

\usepackage{enumerate}
\usepackage[ruled,vlined,linesnumbered]{algorithm2e}

\usepackage[english]{babel}
\usepackage{diagbox,comment,xparse}

\usetikzlibrary{calc}

\usepackage{amsmath}

\newtheorem{theorem}{Theorem}
\newtheorem{lemma}[theorem]{Lemma}

\newtheorem{corollary}{Corollary}
\newtheorem{observation}{Observation}
\newtheorem{proposition}[theorem]{Proposition}

\newtheorem{problem}{Problem}

\newcommand{\QED}{$\Box$}
\newcommand{\cp}{\mathbin{\Box}}

\newcommand{\cH}{\mathcal{H}}

\newcommand{\cT}{\mathcal{T}}

\newcommand{\cN}{{\cal N}}

\newcommand{\cub}{{\rm cubic}}

\newcommand{\spq}{{\sigma_{(p,q)}}}

\newcommand{\proof}{\noindent\textbf{Proof. }}

\newcommand{\1}{ \vspace{0.1cm} }

\let\oldenumerate\enumerate
\renewcommand{\enumerate}{
  \oldenumerate
  \setlength{\itemsep}{0pt}
  \setlength{\parskip}{0pt}
  \setlength{\parsep}{0pt}
}

\begin{document}


\title{Spreading in claw-free cubic graphs}

\author{Bo\v{s}tjan Bre\v{s}ar$^{a,b}$, Jaka Hed\v zet$^{a,b}$, \, and \,  Michael A. Henning$^{c}$ \\ \\
$^a$ Faculty of Natural Sciences and Mathematics \\
University of Maribor, Slovenia \\
$^b$ Institute of Mathematics, Physics and Mechanics, Ljubljana, Slovenia\\
\small \tt Email: bostjan.bresar@um.si \\
\small \tt Email: jaka.hedzet@imfm.com \\
\\
$^{c}$ Department of Mathematics and Applied Mathematics \\
University of Johannesburg, South Africa\\
\small \tt Email: mahenning@uj.ac.za
}

\date{}
\maketitle

\begin{abstract}
Let $p \in \mathbb{N}$ and $q \in \mathbb{N} \cup \lbrace \infty \rbrace$. We study a dynamic coloring of the vertices of a graph $G$ that starts with an initial subset $S$ of blue vertices, with all remaining vertices colored white. If a white vertex~$v$ has at least~$p$ blue neighbors and at least one of these blue neighbors of~$v$ has at most~$q$ white neighbors, then by the spreading color change rule the vertex~$v$ is recolored blue. The initial set $S$ of blue vertices is a $(p,q)$-spreading set for $G$ if by repeatedly applying the spreading color change rule all the vertices of $G$ are eventually colored blue. The $(p,q)$-spreading set is a generalization of the well-studied concepts of $k$-forcing and $r$-percolating sets in graphs. For $q \ge 2$, a $(1,q)$-spreading set is exactly a $q$-forcing set, and the $(1,1)$-spreading set is a $1$-forcing set (also called a zero forcing set), while for $q = \infty$, a $(p,\infty)$-spreading set is exactly a $p$-percolating set. The $(p,q)$-spreading number, $\spq(G)$, of $G$ is the minimum cardinality of a $(p,q)$-spreading set. In this paper, we study $(p,q)$-spreading in claw-free cubic graphs.  While the zero-forcing number of claw-free cubic graphs was studied earlier, for each pair of values $p$ and $q$ that are not both $1$ we either determine the $(p,q)$-spreading number of a claw-free cubic graph $G$ or show that $\sigma_{(p,q)}(G)$ attains one of two possible values. 
\end{abstract}


\section{Introduction}

In this paper we continue the study of dynamic graph colorings and explore the concept of $(p,q)$-spreading in claw-free cubic graph. The concept of spreading in graphs is a generalization of the well-studied concepts of $k$-forcing and $r$-percolating sets in graphs.

Consider a dynamic coloring of the vertices of a graph $G$ that starts with an initial subset $S$ of blue vertices, with all remaining vertices colored white. For $k \in \mathbb{N}$, the color change rule in $k$-forcing is defined as follows: \emph{if a blue vertex $u$ has at most~$k$ white neighbors, then all white neighbors of $u$ are recolored blue}. In particular, when $k=1$, this is the color change rule for zero forcing in graphs. The initial set $S$ of blue vertices is a $k$-forcing set for $G$ if by repeatedly applying the color change rule in $k$-forcing all the vertices of $G$ are eventually colored blue. The $k$-\emph{forcing number} of $G$, denoted $F_k(G)$, is the minimum cardinality among all $k$-forcing set of $G$. When $k = 1$, the $k$-forcing number of $G$ is called the \emph{zero forcing number} and is denoted by $Z(G)$, and so $Z(G) = F_1(G)$.

The concept of bootstrap percolation has a similar definition to that of zero forcing, yet a different motivation. It was introduced as a simplified model of a magnetic system in 1979~\cite{CLR-1979}, and was later studied on random graphs and also on deterministic graphs. In particular, already some early papers considered bootstrap percolation in grids~\cite{BP-1998,Bol-2006}.

As before, we consider a dynamic coloring of the vertices of a graph $G$ that starts with an initial subset $S$ of blue vertices, with all remaining vertices colored white. We refer to blue vertices as ``infected'' and white vertices as ``uninfected''. For $r \in \mathbb{N}$, the color change rule in $r$-percolation is defined as follows: \emph{if a white (uninfected) vertex $u$ has at least $r$ blue (infected) neighbors, then the vertex $u$ is recolored blue.} The initial set $S$ of blue vertices is an \emph{$r$-neighbor bootstrap percolating set}, or simply an \emph{$r$-percolating set}, of $G$ if by repeatedly applying the color change rule in $r$-percolation all the vertices of $G$ are eventually colored blue. The \emph{$r$-neighbor bootstrap percolation number}, or simply the \emph{$r$-percolation number}, of $G$, denoted $m(G,r)$, is the minimum cardinality among all $r$-neighbor bootstrap percolating sets of $G$.

Motivated by the similarity of the definitions of the above two concepts, a common generalization of $k$-forcing and $r$-bootstrap percolation was introduced in~\cite{bdeh}. Let $p \in \mathbb{N}$ and $q \in \mathbb{N} \cup \lbrace \infty \rbrace$. As before, we consider a dynamic coloring of the vertices of a graph $G$ that starts with an initial subset $S$ of blue vertices, with all remaining vertices colored white. The color change rule in $(p,q)$-spreading is defined as follows: \emph{if a white vertex $w$ has at least $p$ blue neighbors, and one of the blue neighbors of~$w$ has at most~$q$ white neighbors, then the vertex $w$ is recolored blue.} The initial set $S$ of blue vertices is a $(p,q)$-\emph{spreading set} of $G$, if by repeatedly applying the $(p,q)$-spreading color change rule all the vertices of $G$ are eventually colored blue. The $(p,q)$-\emph{spreading number}, denoted $\spq(G)$, of a graph $G$ is the minimum cardinality among all $(p,q)$-spreading sets.

The $(p,q)$-spreading set is a generalization of the well-studied concepts of $k$-forcing and $r$-percolating sets in graphs. For $q \ge 2$, a $(1,q)$-spreading set is exactly a $q$-forcing set~\cite{acdp}, and the $(1,1)$-spreading set is a $1$-forcing set (also called a zero forcing set~\cite{AIM,HLS}), while if $q\ge \Delta(G)$ (including $q = \infty$), then a $(p,q)$-spreading set is exactly a $p$-percolating set. In the foundational paper~\cite{bdeh}, the complexity of the decision version of the $(p,q)$-spreading number was proved to be NP-complete for all $p$ and $q$, while efficient algorithms for determining these numbers were found in trees. In addition, for almost all values of $p$ and $q$, the $(p,q)$-spreading numbers of Cartesian grids $P_n\cp P_m$ were established in~\cite{bdeh}. 

The following trivial observation will be (at least implicitly) used several times in the paper.
\begin{observation}
\label{obs:basic}
 Let $G$ be a graph, $p\ge 2$ and $q \in \mathbb{N} \cup \lbrace \infty \rbrace$. If $P$ is a $(p,q)$-spreading set in $G$, then $P$ is also a $(p,q+1)$-spreading set in $G$ as well as a $(p-1,q)$-spreading set in $G$. In particular,
 $$\spq(G)\ge \sigma_{(p,q+1)}(G) \textrm{ and }  \spq(G)\ge \sigma_{(p-1,q)}(G).$$
\end{observation}

A graph is \emph{claw}-\emph{free} if it does not contain $K_{1,3}$ as an induced subgraph. A \emph{cubic graph} (also called a $3$-\emph{regular graph}) is a graph in which every vertex has degree~$3$. In this paper we study spreading in claw-free cubic graphs.

\subsection{Graph theory notation and terminology}
\label{S:notation}

For notation and graph theory terminology, we in general follow~\cite{HaHeHe-23}. Specifically, let $G$ be a graph with vertex set $V(G)$ and edge set $E(G)$, and of order~$n(G) = |V(G)|$ and size $m(G) = |E(G)|$. A \emph{neighbor} of a vertex $v$ in $G$ is a vertex $u$ that is adjacent to $v$, that is, $uv \in E(G)$. The \emph{open neighborhood} $N_G(v)$ of a vertex $v$ in $G$ is the set of neighbors of $v$, while the \emph{closed neighborhood} of $v$ is the set $N_G[v] = \{v\} \cup N(v)$. We denote the \emph{degree} of $v$ in $G$ by $\deg_G(v) = |N_G(v)|$, and $\Delta(G)=\max\{\deg_G(v):\, v\in V(G)\}$. For a set $S \subseteq V(G)$, its \emph{open neighborhood} is the set $N_G(S) = \cup_{v \in S} N_G(v)$, and its \emph{closed neighborhood} is the set $N_G[S] = N_G(S) \cup S$.

For a set $S \subseteq V(G)$, the subgraph induced by $S$ is denoted by $G[S]$. Further, the subgraph of $G$ obtained from $G$ by deleting all vertices in $S$ and all edges incident with vertices in $S$ is denoted by $G - S$; that is, $G-S = G[V(G)\setminus S]$. If $S = \{v\}$, then we also denote $G - S$ simply by $G - v$. If $F$ is a graph, then an $F$-\emph{component} of $G$ is a component isomorphic to~$F$. We denote by $\alpha(G)$ the independence number of $G$, and so $\alpha(G)$ is the maximum cardinality among all independent sets of $G$. A \emph{vertex cover} of $G$ is a set $S$ of vertices such that every edge of $G$ is incident with at least one vertex in $S$. The \emph{vertex covering number} $\beta(G)$ (also denoted $\tau(G)$ in the literature), equals the minimum cardinality among all vertex covers of $G$. By the well-known Gallai theorem, $\alpha(G)+\beta(G)=n(G)$ holds in any graph $G$.

We denote the \emph{path}, \emph{cycle}, and \emph{complete graph} on $n$ vertices by $P_n$, $C_n$, and $K_n$, respectively, and we denote the \emph{complete bipartite graph} with partite sets of cardinality~$n$ and $m$ by $K_{n,m}$. A \emph{triangle} in $G$ is a subgraph isomorphic to $K_3$, whereas a \emph{diamond} in $G$ is an induced subgraph of $G$ isomorphic to $K_4$ with one edge missing, denoted by $K_4 - e$. A graph is \emph{diamond}-\emph{free} if it does not contain $K_4-e$ as an induced subgraph. For $k \ge 1$ an integer, we use the standard notation $[k] = \{1,\ldots,k\}$.

\subsection{Main results and organization of the paper}

Davila and Henning studied zero forcing in claw-free cubic graphs~\cite{dahe-2020}. They established two upper bounds, which are summarized in the following result (for the definition of diamond-necklace $N_k$, see Section~\ref{sec:CFcubic}).

\begin{theorem}{\rm (\cite{dahe-2020})}
\label{thm:davila-henning}
If $G\ne K_4$ is a connected, claw-free, cubic graph of order $n$, then the following properties hold: \\ [-22pt]
\begin{enumerate}
\item[{\rm (a)}] $Z(G) \le \alpha(G) + 1$; \1
\item[{\rm (b)}]  $Z(G) \le \frac{n}{3} + 1$, unless $G = N_2$ in which case $Z(G) = \frac{1}{3} (n + 4)$.
\end{enumerate} 
\end{theorem}

It is proved in~\cite{dahe-2020} that both upper bounds in Theorem~\ref{thm:davila-henning} are sharp. He et al.~\cite{he2024} recently characterized the connected, claw-free cubic graphs $G$ for which $Z(G)=\alpha(G) + 1$. Notably, there are only three sporadic graphs that attain this value, namely $N_2$, $N_3$ and the Hamming graph $K_3\Box K_2$. Since these graphs have $8, 12$ and $6$ vertices, respectively, the following result immediately follows. 
\begin{theorem}{\rm (\cite{he2024})}
\label{thm:he}
If $G\ne K_4$ is a connected, claw-free, cubic graph of order at least~$14$, then $Z(G) \le \alpha(G)$.
\end{theorem}

Note that $Z(G)=\sigma_{(1,1)}(G)$, and the mentioned result (Theorem~\ref{thm:davila-henning}(a)) is placed in the $(1,1)$-entry of Table~\ref{Tabela1}.

From properties of connected, claw-free, cubic graphs $G$ which we discuss in Section~\ref{sec:CFcubic}, if $G \ne K_4$ then there is a unique partition of the vertex set $V(G)$ into subsets each of which induces either a triangle or a diamond. The number of these subsets, called units, in $G$ is denoted by $u(G)$. Table~\ref{Tabela1} summarizes the values of $(p,q)$-spreading numbers in claw-free cubic graphs $G \ne K_4$ for all possible $p$ and $q$ in terms of the parameters $\alpha(G)$, $\beta(G)$, $u(G)$ and $n(G)$. Besides the $(1,1)$-entry giving the mentioned upper bound on $\sigma_{(1,1)}(G)$ all other values in Table~\ref{Tabela1} are obtained in this paper. It is easy to see that $\sigma_{(1,2)}(G)=2$ for any claw-free cubic graph $G$. Indeed, any adjacent pair of vertices in $G$ is a $(1,2)$-spreading set in $G$. The results $\sigma_{(1,q)}(G)=1$ for any $q\ge 3$ and $\sigma_{(p,q)}(G)=n(G)$ for any $p\ge 4$ and any $q\in\mathbb{N}\cup\{\infty\}$ are trivial consequences of definitions.

\begin{table}[ht!]
\begin{center}
\begin{tabular}{| c | c |c | c | }
\hline
\diagbox{$p$}{$q$} & 1 & 2 & $\ge 3$   \\
 \hline
1 & $\le \alpha(G)(+1)$, $n\ge 14$ {\scriptsize [Thm~\ref{thm:he}]} {\scriptsize ([Thm~\ref{thm:davila-henning}])}  & 2 & 1                                          \\ \hline
2 &  $u(G)+1$ or $u(G)+2$ \hskip33pt {\scriptsize [Thm~\ref{thm:s21}]}    & $u(G)$ or $u(G)+1$ \hskip9pt {\scriptsize [Cor~\ref{cor:sigma22-23}]} & $u(G)$ or $u(G)+1$  \hskip8pt {\scriptsize [Cor~\ref{cor:2percolation}]}                                          \\ \hline
3 & $\beta(G)$ or $\beta(G)+1$ \hskip44pt {\scriptsize [Prp~\ref{prp:s31}]} & \hskip45pt $\beta(G)$ \hskip25pt {\scriptsize [Prp~\ref{prop:sigma32}]} & \hskip45pt $\beta(G)$ \hskip25pt {\scriptsize [Prp~\ref{prop:sigma32}]} \\ \hline
$\ge 4$ & \hskip -1.5cm $n(G)$ & \hskip -5pt  $n(G)$ & \hskip -6pt  $n(G)$ \\ \hline
\end{tabular}
\caption{$(p,q)$-spreading numbers of claw-free cubic graphs $G$, $G\ne K_4$. }
\label{Tabela1}
\end{center}
\end{table}

The paper is organized as follows. In Section~\ref{sec:3-percolation}, we consider $(3,q)$-spreading numbers in claw-free cubic graphs; their values can be found in Line 3 of Table~\ref{Tabela1}. (Note that $3$-percolation coincides with $(3,q)$-spreading when $q\ge \Delta(G)$.) In Section~\ref{sec:2-percolation}, we deal with $(2,q)$-spreading numbers of claw-free cubic graphs; see Line 2 in Table~\ref{Tabela1}. In the case of $(2,q)$-spreading, where $q\ge 3$, $\sigma_{(2,q)}(G)=u(G)+1$ precisely when $G$ is a diamond necklace $N_k$, while in all other claw-free cubic graphs $\sigma_{(2,q)}(G)=u(G)$. Thus the $2$-percolation number of claw-free cubic graphs is fully determined. On the other hand, when $q\le 2$ a claw-free cubic graph can attain two values: $\sigma_{(2,2)}(G)\in\{u(G),u(G)+1\}$, and $\sigma_{(2,1)}(G)\in\{u(G)+1,u(G)+2\}$, but we have not established exactly which claw-free cubic graphs attain which of the two possible values. We conclude the paper with some open problems that arise from this study.

\section{Claw-free cubic graphs}
\label{sec:CFcubic}

The following property of connected, claw-free, cubic graphs is established in~\cite{HeLo-12}.

\begin{lemma}{\rm (\cite{HeLo-12})}
\label{lem:known}
If $G \ne K_4$ is a connected, claw-free, cubic graph of order $n$, then the vertex set $V(G)$ can be uniquely partitioned into sets each of which induces a triangle or a diamond in $G$.
\end{lemma}

By Lemma~\ref{lem:known}, the vertex set $V(G)$ of connected, claw-free, cubic graph $G \ne K_4$  can be uniquely partitioned into sets each of which induce a triangle or a diamond in $G$. Following the notation introduced in~\cite{HeLo-12}, we refer to such a partition as a \emph{triangle}-\emph{diamond partition} of $G$, abbreviated $\Delta$-D-partition. We call every triangle and diamond induced by a set in our $\Delta$-D-partition a \emph{unit} of the partition. A unit that is a triangle is called a \emph{triangle-unit} and a unit that is a diamond is called a \emph{diamond-unit}. (We note that a triangle-unit is a triangle that does not belong to a diamond.) Two units in the $\Delta$-D-partition are \emph{adjacent} if there is an edge joining a vertex in one unit to a vertex in the other unit. In what follows, we define several structures in claw-free, cubic graphs that we will need when proving our main results. Some of these definitions have already appeared in the literature, see, for example,~\cite{BaHe-22a,BaHe-22b,HeLo-12}.

For $k \ge 2$ an integer, let $N_k$ be the connected cubic graph constructed as follows. Take $k$ disjoint copies $D_1, D_2, \ldots, D_{k}$ of a diamond, where $V(D_i) = \{a_i,b_i,c_i,d_i\}$ and where $a_ib_i$ is the missing edge in $D_i$. Let $N_k$ be obtained from the disjoint union of these $k$ diamonds by adding the edges $\{a_ib_{i+1} \colon i \in [k-1] \}$  and adding the edge $a_{k}b_1$. We call $N_k$ a \emph{diamond}-\emph{necklace} with $k$ diamonds. Let $\cN_\cub = \{ N_k \mid k \ge 2\}$. A diamond-necklace, $N_4$, with four diamonds is illustrated in Figure~\ref{fig:N4}.

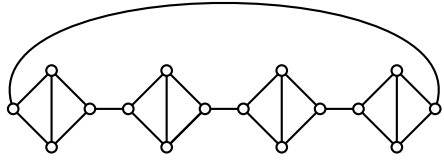
\begin{figure}[htb]
\begin{center}
\begin{tikzpicture}[scale=.8,style=thick,x=.85cm,y=.85cm]
\def\vr{2.5pt} 

%
\path (2.25,0.5) coordinate (a);
\path (2.25,2) coordinate (b);
\path (1.5,1.25) coordinate (c);
\path (3,1.25) coordinate (d);
\path (4.5,0.5) coordinate (w);
\path (4.5,2) coordinate (x);
\path (3.75,1.25) coordinate (y);
\path (5.25,1.25) coordinate (z);
\path (6.75,0.5) coordinate (k);
\path (6.75,2) coordinate (l);
\path (6,1.25) coordinate (m);
\path (7.5,1.25) coordinate (n);
\path (8.25,1.25) coordinate (h);
\path (9,0.5) coordinate (i);
\path (9,2) coordinate (j);
\path (9.75,1.25) coordinate (o);
\draw (a) -- (c);
\draw (a) -- (d);
\draw (b) -- (c);
\draw (b) -- (d);
\draw (a) -- (b);
%
\draw (d) -- (y);
\draw (y)--(w)--(z)--(w)--(x)--(y);
\draw (x)--(z);
\draw (h) -- (i);
\draw (h) -- (j);
\draw (i) -- (j);
\draw (k) -- (m);
\draw (k) -- (n);
\draw (l) -- (m);
\draw (l) -- (n);
\draw (k) -- (l);
\draw (m) -- (z);
\draw (n) -- (h);
\draw (i)--(o)--(j);
\draw (c) to[out=110,in=70, distance=2.5cm] (o);
\draw (a) [fill=white] circle (\vr);
\draw (b) [fill=white] circle (\vr);
\draw (c) [fill=white] circle (\vr);
\draw (d) [fill=white] circle (\vr);
\draw (h) [fill=white] circle (\vr);
\draw (i) [fill=white] circle (\vr);
\draw (j) [fill=white] circle (\vr);
\draw (k) [fill=white] circle (\vr);
\draw (l) [fill=white] circle (\vr);
\draw (m) [fill=white] circle (\vr);
\draw (n) [fill=white] circle (\vr);
\draw (w) [fill=white] circle (\vr);
\draw (x) [fill=white] circle (\vr);
\draw (y) [fill=white] circle (\vr);
\draw (z) [fill=white] circle (\vr);
\draw (o) [fill=white] circle (\vr);

\draw[anchor = north] (a) node {{\small $c_1$}};
\draw[anchor = south] (b) node {{\small $d_1$}};
\draw[anchor = north] (c) node {{\small $b_1$}};
\draw[anchor = north] (d) node {{\small $a_1$}};

\draw[anchor = north] (w) node {{\small $c_2$}};
\draw[anchor = south] (x) node {{\small $d_2$}};
\draw[anchor = north] (y) node {{\small $b_2$}};
\draw[anchor = north] (z) node {{\small $a_2$}};

\draw[anchor = north] (k) node {{\small $c_3$}};
\draw[anchor = south] (l) node {{\small $d_3$}};
\draw[anchor = north] (m) node {{\small $b_3$}};
\draw[anchor = north] (n) node {{\small $a_3$}};

\draw[anchor = south] (j) node {{\small $d_4$}};
\draw[anchor = north] (i) node {{\small $c_4$}};
\draw[anchor = north] (h) node {{\small $b_4$}};
\draw[anchor = north] (o) node {{\small $a_4$}};
%
\end{tikzpicture}
\end{center}
\vskip -0.5cm
\caption{A diamond-necklace $N_4$} \label{fig:N4}
\end{figure}

For $k \ge 1$ an integer, let $F_{2k}$ be the connected cubic graph constructed as follows. Take $2k$ disjoint copies $T_1, T_2, \ldots, T_{2k}$ of a triangle, where $V(T_i) = \{x_i,y_i,z_i\}$ for $i \in [2k]$. Let
\[
\begin{array}{lcl}
E_a & = & \{x_{2i-1}x_{2i} \colon i \in [k] \} \1 \\
E_b & = & \{y_{2i-1}y_{2i} \colon i \in [k] \} \1 \\
E_c & = & \{z_{2i}z_{2i+1} \colon i \in [k] \},
\end{array}
\]
where addition is taken modulo~$2k$ (and so, $z_1 = z_{2k+1}$). Let $F_{2k}$ be obtained from the disjoint union of these $2k$ triangles by adding the edges $E_a \cup E_b \cup E_c$. The resulting graph $F_{2k}$ we call a \emph{triangle}-\emph{necklace} with $2k$ triangles. Let $\cT_\cub = \{ F_{2k} \colon k \ge 1\}$. The triangle-necklace $F_6$ is shown in Figure~\ref{fig:Tneck}.

\begin{figure}[htb]
\begin{center}
\begin{tikzpicture}[scale=.8,style=thick,x=.85cm,y=.85cm]
\def\vr{2.5pt} 
\path (2,0.75) coordinate (a);
\path (3,0) coordinate (b);
\path (3,1.5) coordinate (c);
\path (4,0) coordinate (d);
\path (4,1.5) coordinate (e);
\path (5,0.75) coordinate (f);
\path (6,0.75) coordinate (g);
\path (7,0) coordinate (h);
\path (7,1.5) coordinate (i);
\path (8,0) coordinate (j);
\path (8,1.5) coordinate (k);
\path (9,0.75) coordinate (l);
\path (10,0.75) coordinate (m);
\path (11,0) coordinate (n);
\path (11,1.5) coordinate (o);
\path (12,0) coordinate (p);
\path (12,1.5) coordinate (q);
\path (13,0.75) coordinate (r);
\draw (b)--(a)--(c)--(b)--(d)--(e)--(f)--(d);
\draw (c)--(e);
\draw (f)--(g);
\draw (i)--(g)--(h)--(i)--(k)--(l)--(j)--(k);
\draw (n)--(m)--(o)--(n)--(p)--(r)--(q)--(p);
\draw (h)--(j);
\draw (l)--(m);
\draw (o)--(q);
%
\draw (a) to[out=110,in=70, distance=2.5cm] (r);
\draw (a) [fill=white] circle (\vr);
\draw (b) [fill=white] circle (\vr);
\draw (c) [fill=white] circle (\vr);
\draw (d) [fill=white] circle (\vr);
\draw (e) [fill=white] circle (\vr);
\draw (f) [fill=white] circle (\vr);
\draw (g) [fill=white] circle (\vr);
\draw (h) [fill=white] circle (\vr);
\draw (i) [fill=white] circle (\vr);
\draw (j) [fill=white] circle (\vr);
\draw (k) [fill=white] circle (\vr);
\draw (l) [fill=white] circle (\vr);
\draw (m) [fill=white] circle (\vr);
\draw (n) [fill=white] circle (\vr);
\draw (o) [fill=white] circle (\vr);
\draw (p) [fill=white] circle (\vr);
\draw (q) [fill=white] circle (\vr);
\draw (r) [fill=white] circle (\vr);
\draw[anchor = north] (a) node {{\small $x_1$}};
\draw[anchor = north] (b) node {{\small $y_1$}};
\draw[anchor = south] (c) node {{\small $z_1$}};
\draw[anchor = north] (d) node {{\small $y_2$}};
\draw[anchor = south] (e) node {{\small $z_2$}};
\draw[anchor = north] (f) node {{\small $x_2$}};

\draw[anchor = north] (g) node {{\small $x_3$}};
\draw[anchor = north] (h) node {{\small $y_3$}};
\draw[anchor = south] (i) node {{\small $z_3$}};
\draw[anchor = north] (j) node {{\small $y_4$}};
\draw[anchor = south] (k) node {{\small $z_4$}};
\draw[anchor = north] (l) node {{\small $x_4$}};

\draw[anchor = north] (m) node {{\small $x_5$}};
\draw[anchor = north] (n) node {{\small $y_5$}};
\draw[anchor = south] (o) node {{\small $z_5$}};
\draw[anchor = north] (p) node {{\small $y_6$}};
\draw[anchor = south] (q) node {{\small $z_6$}};
\draw[anchor = north] (r) node {{\small $x_6$}};

\end{tikzpicture}
\end{center}
\vskip -0.35cm
\caption{The triangle-necklace $F_6$} \label{fig:Tneck}
\end{figure}
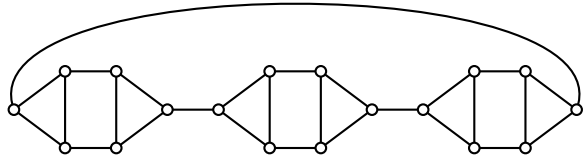

For $k \ge 2$ an integer, take $2k$ disjoint copies $T_1, T_2, \ldots, T_{2k}$ of a triangle, where $V(T_i) = \{x_i,y_i,z_i\}$ for $i \in [2k]$, and take $k$ disjoint copies $D_1, D_2, \ldots, D_{k}$ of a diamond, where $V(D_j) = \{a_j,b_j,c_j,d_j\}$ and where $a_jb_j$ is the missing edge in $D_j$ for $j \in [k]$. Let
\[
\begin{array}{lcl}
E_1 & = & \{x_{2i-1}a_{i} \colon i \in [k] \} \1 \\
E_2 & = & \{x_{2i}b_{i} \colon i \in [k] \} \1 \\
E_3 & = & \{y_{2i-1}z_{2i+1} \colon i \in [k-1] \} \cup \{y_{2k-1}z_1\}, \1 \\
E_4 & = & \{y_{2i}z_{2i+2} \colon i \in [k-1] \} \cup \{y_{2k}z_{2}\}.
\end{array}
\]
Let $H_{2k}$ be obtained from the disjoint union of these $2k$ triangles and $k$ diamonds by adding the edges $E_1 \cup E_2 \cup E_3 \cup E_4$. The resulting graph $H_{2k}$ we call a \emph{triangle}-\emph{diamond}-\emph{necklace} with $k$ diamonds. Let $\cH_\cub = \{ H_{2k} \colon k \ge 2\}$. The triangle-diamond-necklace $H_6$ is shown in Figure~\ref{fig:TDneck}.

\begin{figure}[htb]
\begin{center}
\begin{tikzpicture}[scale=.8,style=thick,x=.85cm,y=.85cm]
\def\vr{2.5pt} 
\path (4,0) coordinate (d);
\path (4,1.5) coordinate (e);
\path (5,0.75) coordinate (f);
\path (6,0.75) coordinate (g);
\path (7,0) coordinate (h);
\path (7,1.5) coordinate (i);

\path (8,0.75) coordinate (l);
\path (9,0.75) coordinate (m);
\path (10,0) coordinate (n);
\path (10,1.5) coordinate (o);
\path (4,2.5) coordinate (d1);
\path (4,4) coordinate (e1);
\path (5,3.25) coordinate (f1);
\path (6,3.25) coordinate (g1);
\path (7,2.5) coordinate (h1);
\path (7,4) coordinate (i1);

\path (8,3.25) coordinate (l1);
\path (9,3.25) coordinate (m1);
\path (10,2.5) coordinate (n1);
\path (10,4) coordinate (o1);
\path (4,5) coordinate (d2);
\path (4,6.5) coordinate (e2);
\path (5,5.75) coordinate (f2);
\path (6,5.75) coordinate (g2);
\path (7,5) coordinate (h2);
\path (7,6.5) coordinate (i2);

\path (8,5.75) coordinate (l2);
\path (9,5.75) coordinate (m2);
\path (10,5) coordinate (n2);
\path (10,6.5) coordinate (o2);
\draw (d)--(e)--(f)--(d);
\draw (f)--(g);
\draw (i)--(g)--(h)--(i)--(l)--(h);
\draw (n)--(m)--(o)--(n);

\draw (l)--(m);
\draw (d1)--(e1)--(f1)--(d1);
\draw (f1)--(g1);
\draw (i1)--(g1)--(h1)--(i1)--(l1)--(h1);
\draw (n1)--(m1)--(o1)--(n1);

\draw (l1)--(m1);
\draw (d2)--(e2)--(f2)--(d2);
\draw (f2)--(g2);
\draw (i2)--(g2)--(h2)--(i2)--(l2)--(h2);
\draw (n2)--(m2)--(o2)--(n2);

\draw (l2)--(m2);
\draw (e)--(d1);
\draw (e1)--(d2);
\draw (o)--(n1);
\draw (o1)--(n2);
\draw (d) to[out=180,in=180, distance=1.5cm] (e2);
\draw (n) to[out=0,in=0, distance=1.5cm] (o2);
\draw (d) [fill=white] circle (\vr);
\draw (e) [fill=white] circle (\vr);
\draw (f) [fill=white] circle (\vr);
\draw (g) [fill=white] circle (\vr);
\draw (h) [fill=white] circle (\vr);
\draw (i) [fill=white] circle (\vr);

\draw (l) [fill=white] circle (\vr);
\draw (m) [fill=white] circle (\vr);
\draw (n) [fill=white] circle (\vr);
\draw (o) [fill=white] circle (\vr);
\draw (d1) [fill=white] circle (\vr);
\draw (e1) [fill=white] circle (\vr);
\draw (f1) [fill=white] circle (\vr);
\draw (g1) [fill=white] circle (\vr);
\draw (h1) [fill=white] circle (\vr);
\draw (i1) [fill=white] circle (\vr);

\draw (l1) [fill=white] circle (\vr);
\draw (m1) [fill=white] circle (\vr);
\draw (n1) [fill=white] circle (\vr);
\draw (o1) [fill=white] circle (\vr);
\draw (d2) [fill=white] circle (\vr);
\draw (e2) [fill=white] circle (\vr);
\draw (f2) [fill=white] circle (\vr);
\draw (g2) [fill=white] circle (\vr);
\draw (h2) [fill=white] circle (\vr);
\draw (i2) [fill=white] circle (\vr);

\draw (l2) [fill=white] circle (\vr);
\draw (m2) [fill=white] circle (\vr);
\draw (n2) [fill=white] circle (\vr);
\draw (o2) [fill=white] circle (\vr);

\draw[anchor = north] (d) node {{\small $z_1$}};
\draw[anchor = east] (e) node {{\small $y_1$}};
\draw[anchor = south] (f) node {{\small $x_1$}};

\draw[anchor = north] (g) node {{\small $a_1$}};
\draw[anchor = west] (h) node {{\small $d_1$}};
\draw[anchor = east] (i) node {{\small $c_1$}};
\draw[anchor = south] (l) node {{\small $b_1$}};

\draw[anchor = south] (m) node {{\small $x_2$}};
\draw[anchor = north] (n) node {{\small $z_2$}};
\draw[anchor = west] (o) node {{\small $y_2$}};

\draw[anchor = east] (d1) node {{\small $z_3$}};
\draw[anchor = east] (e1) node {{\small $y_3$}};
\draw[anchor = south] (f1) node {{\small $x_3$}};

\draw[anchor = north] (g1) node {{\small $a_2$}};
\draw[anchor = west] (h1) node {{\small $d_2$}};
\draw[anchor = east] (i1) node {{\small $c_2$}};
\draw[anchor = south] (l1) node {{\small $b_2$}};

\draw[anchor = south] (m1) node {{\small $x_4$}};
\draw[anchor = west] (n1) node {{\small $z_4$}};
\draw[anchor = west] (o1) node {{\small $y_4$}};

\draw[anchor = east] (d2) node {{\small $z_5$}};
\draw[anchor = south] (e2) node {{\small $y_5$}};
\draw[anchor = south] (f2) node {{\small $x_5$}};

\draw[anchor = north] (g2) node {{\small $a_3$}};
\draw[anchor = west] (h2) node {{\small $d_3$}};
\draw[anchor = east] (i2) node {{\small $c_3$}};
\draw[anchor = south] (l2) node {{\small $b_3$}};

\draw[anchor = south] (m2) node {{\small $x_6$}};
\draw[anchor = west] (n2) node {{\small $z_6$}};
\draw[anchor = south] (o2) node {{\small $y_6$}};

\end{tikzpicture}
\end{center}
\vskip -0.35cm
\caption{The triangle-diamond-necklace $H_6$} \label{fig:TDneck}
\end{figure}
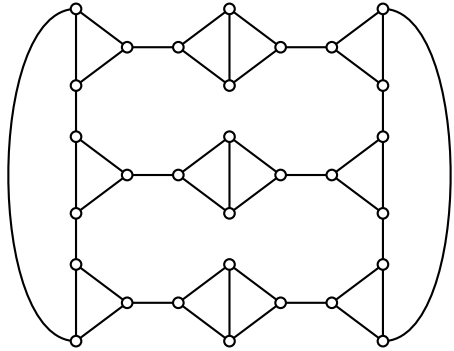

\section{$3$-Percolation in claw-free, cubic graphs}
\label{sec:3-percolation}

In this section, we study $3$-percolation in claw-free, cubic graphs. The following lemma is already known in the literature and follows readily from the definition of $r$-bootstrap percolation.

\begin{lemma}{\rm (\cite{HH-2024+})}
\label{lem:subgraph condition}
For $r \ge 2$ if $H$ is a subgraph of a graph $G$ such that every vertex in $H$ has strictly less than~$r$ neighbors in $G$ that belong to $V(G) \setminus V(H)$, then every $r$-percolating set of $G$ contains at least one vertex of $H$.
\end{lemma}

As a consequence of Lemma~\ref{lem:subgraph condition}, we establish a relationship between the $r$-percolation number and the independence number of an $r$-regular graph. Recall that $\beta(G)$ is the cardinality of a minimum vertex cover of $G$.

\begin{proposition}
\label{prop:relate}
Let $G$ be an $r$-regular graph of order~$n$, where $r \ge 2$. A set $S\subset V(G)$ is an $r$-percolating set in $G$ if and only if $S$ is a vertex cover in $G$. In particular, $m(G,r) = \beta(G) = n - \alpha(G)$.
\end{proposition}
\proof
Let $S$ be an $r$-percolating set of an $r$-regular graph $G$. Suppose that $G-S$ contains an edge $uv$. Let $H = G[\{u,v\}]$, and so $H = K_2$ is a subgraph of $G$ such that every vertex in $H$ has strictly less than~$r$ neighbors in $G$ that belong to $V(G) \setminus V(H)$. By Lemma~\ref{lem:subgraph condition}, we infer that every $r$-percolating set of $G$ contains at least one vertex of $H$, which is a contradiction with our assumption. Thus, $S$ is a vertex cover of $G$. Conversely, suppose that $S$ is a vertex cover of $G$, and let $I = V(G) \setminus S$. Since the set $I$ is an independent set of $G$, every vertex in $I$ has all its $r$ neighbors in the set $V(G) \setminus I$. The set $S$ is therefore an $r$-percolating set of $G$. Since $m(G,r)$ is the cardinality of a minimum $r$-percolating set, we infer that $m(G,r) = \beta(G)$. Consequently, by the famous Gallai theorem ($\alpha(G)+\beta(G)=n(G)$ in any graph $G$) we also have $m(G,r) = n - \alpha(G)$.~\QED

\medskip

Let us recall the statement of the well-known Brooks' Coloring Theorem. 

\begin{theorem}{\rm (Brooks' Coloring Theorem)}
\label{thm:brooks}
If $G$ is a connected graph, which is not a complete graph or an odd cycle, then $\chi(G)\le \Delta(G)$. 
\end{theorem}

As a consequence of Theorem~\ref{thm:brooks}, we have the following trivial lower bound on the independence number of a graph.

\begin{theorem}
\label{thm:lower-bound-alpha}
If $G \ne K_{n}$ is a connected graph of order~$n$ with maximum degree~$\Delta \ge 3$, then $\alpha(G)  \ge \frac{n}{\Delta}$.
\end{theorem}
\proof As a consequence of Theorem~\ref{thm:brooks}, if $G \ne K_n$ is a connected graph of order~$n$ with maximum degree~$\Delta \ge 3$, then the chromatic number of $G$ is at most~$\Delta$, that is, $\chi(G) \le \Delta$. Alternatively, viewing a $\Delta$-coloring of $G$ as a partitioning of its vertices into $\Delta$ independent sets, called color classes, we infer by the Pigeonhole Principle that $G$ contains an independent set of cardinality at least~$n/\Delta$, implying that $\alpha(G) \ge n/\Delta$.~\QED

\medskip
The following upper bound on the independence number of a claw-free graph is given by several authors.

\begin{theorem}{\rm (\cite{LiVi-90,Faudree-92})}
\label{thm:upper-bound-alpha}
If $G$ is a claw-free graph of order~$n$ with minimum degree~$\delta$, then
\[
\alpha(G) \le \left( \frac{2}{\delta + 2} \right) n.
\]
\end{theorem}

As a consequence of Theorems~\ref{thm:lower-bound-alpha} and~\ref{thm:upper-bound-alpha}, we have the following bounds on the independence number of a claw-free graph.

\begin{theorem}
\label{thm:bounds-alpha}
If $G \ne K_{n}$ is a connected, claw-free graph of order~$n$ with minimum degree~$\delta$ and maximum degree~$\Delta \ge 3$, then
\[
\frac{n}{\Delta} \le \alpha(G) \le \left( \frac{2}{\delta + 2} \right) n.
\]
\end{theorem}

As a consequence of Proposition~\ref{prop:relate} and Theorem~\ref{thm:bounds-alpha}, we have the following result.

\begin{theorem}
\label{thm:bounds-alpha-cubic}
If $G \ne K_4$ is a connected, claw-free, cubic graph of order~$n$, then the following properties hold. \\ [-22pt]
\begin{enumerate}
\item[{\rm (a)}] $\frac{1}{3}n  \le \alpha(G) \le   \frac{2}{5}n$. \1
\item[{\rm (b)}] $\frac{3}{5}n  \le m(G,3) \le  \frac{2}{3}n$.
\end{enumerate}
\end{theorem}

We show next that the bounds of Theorem~\ref{thm:bounds-alpha-cubic} are tight (in the sense that they hold for connected graphs of arbitrarily large orders). Suppose that $G \in \cT_\cub$. Thus, $G$ is a triangle-necklace $F_{2k}$ for some $k \ge 1$, and so $G$ has order~$n = 6k$ and contains~$2k$ vertex disjoint triangles. Moreover, every unit of the (unique) $\Delta$-D-partition of $G$ is a triangle-unit, implying that $\alpha(G) \le \frac{1}{3}n$. By Theorem~\ref{thm:bounds-alpha-cubic}(a), $\alpha(G) \ge \frac{1}{3}n$. Consequently, $\alpha(G) = \frac{1}{3}n$. For example, the white vertices in the triangle-necklace $F_6$ of order~$n = 18$ shown in Figure~\ref{fig:Tneck1} form an $\alpha$-set of $F_6$ (of cardinality~$6= \frac{1}{3}n$). By Proposition~\ref{prop:relate}, we infer that $m(G,3) = \frac{2}{3}n$, and the corresponding set of shaded vertices in Figure~\ref{fig:Tneck1} is a $\beta$-set of $G$. This shows that the lower bound of Theorem~\ref{thm:bounds-alpha-cubic}(a) is tight, as is the upper bound of Theorem~\ref{thm:bounds-alpha-cubic}(b).

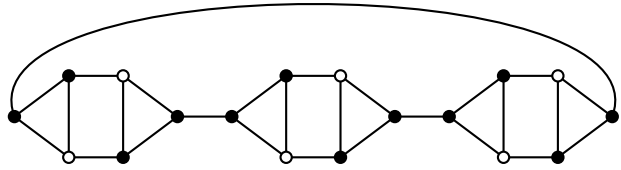
\begin{figure}[htb]
\begin{center}
\begin{tikzpicture}[scale=.85,style=thick,x=.85cm,y=.85cm]
\def\vr{2.5pt} 
\path (2,0.75) coordinate (a);
\path (3,0) coordinate (b);
\path (3,1.5) coordinate (c);
\path (4,0) coordinate (d);
\path (4,1.5) coordinate (e);
\path (5,0.75) coordinate (f);
\path (6,0.75) coordinate (g);
\path (7,0) coordinate (h);
\path (7,1.5) coordinate (i);
\path (8,0) coordinate (j);
\path (8,1.5) coordinate (k);
\path (9,0.75) coordinate (l);
\path (10,0.75) coordinate (m);
\path (11,0) coordinate (n);
\path (11,1.5) coordinate (o);
\path (12,0) coordinate (p);
\path (12,1.5) coordinate (q);
\path (13,0.75) coordinate (r);
\draw (b)--(a)--(c)--(b)--(d)--(e)--(f)--(d);
\draw (c)--(e);
\draw (f)--(g);
\draw (i)--(g)--(h)--(i)--(k)--(l)--(j)--(k);
\draw (n)--(m)--(o)--(n)--(p)--(r)--(q)--(p);
\draw (h)--(j);
\draw (l)--(m);
\draw (o)--(q);
%
\draw (a) to[out=110,in=70, distance=2.5cm] (r);
\draw (a) [fill=black] circle (\vr);
\draw (b) [fill=white] circle (\vr);
\draw (c) [fill=black] circle (\vr);
\draw (d) [fill=black] circle (\vr);
\draw (e) [fill=white] circle (\vr);
\draw (f) [fill=black] circle (\vr);
\draw (g) [fill=black] circle (\vr);
\draw (h) [fill=white] circle (\vr);
\draw (i) [fill=black] circle (\vr);
\draw (j) [fill=black] circle (\vr);
\draw (k) [fill=white] circle (\vr);
\draw (l) [fill=black] circle (\vr);
\draw (m) [fill=black] circle (\vr);
\draw (n) [fill=white] circle (\vr);
\draw (o) [fill=black] circle (\vr);
\draw (p) [fill=black] circle (\vr);
\draw (q) [fill=white] circle (\vr);
\draw (r) [fill=black] circle (\vr);
\end{tikzpicture}
\end{center}
\vskip -0.35cm
\caption{A $\beta$-set in a triangle-necklace $F_6$} \label{fig:Tneck1}
\end{figure}

Suppose next that $G \in \cH_\cub$. Thus, $G$ is a triangle-diamond-necklace $H_{2k}$ for some $k \ge 2$, and so $G$ has order~$n = 10k$. We can choose an independent set of $G$ to contain one vertex from every triangle-unit and two vertices from every diamond-unit, implying that $\alpha(G) \ge 4k = \frac{2}{5}n$. For example, the white vertices in the triangle-diamond-necklace $H_6$ of order~$n = 30$ shown in Figure~\ref{fig:TDneck1} form an $\alpha$-set of $H_6$ (of cardinality~$12= \frac{2}{5}n$). By Proposition~\ref{prop:relate}, we infer that $m(G,3) = \frac{3}{5}n$, and the corresponding set of shaded vertices in Figure~\ref{fig:TDneck1} is a $\beta$-set of $H_6$ . This shows that the upper bound of Theorem~\ref{thm:bounds-alpha-cubic}(a) is tight, as is the lower bound of Theorem~\ref{thm:bounds-alpha-cubic}(b).

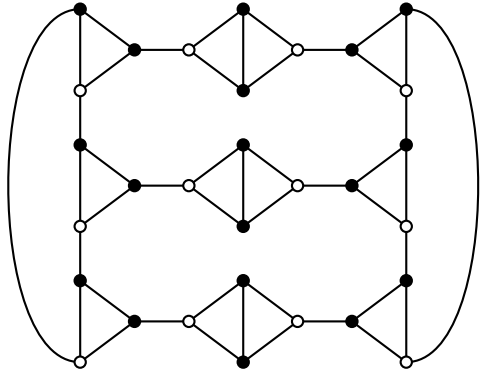
\begin{figure}[htb]
\begin{center}
\begin{tikzpicture}[scale=.85,style=thick,x=.85cm,y=.85cm]
\def\vr{2.5pt} 
\path (4,0) coordinate (d);
\path (4,1.5) coordinate (e);
\path (5,0.75) coordinate (f);
\path (6,0.75) coordinate (g);
\path (7,0) coordinate (h);
\path (7,1.5) coordinate (i);

\path (8,0.75) coordinate (l);
\path (9,0.75) coordinate (m);
\path (10,0) coordinate (n);
\path (10,1.5) coordinate (o);
\path (4,2.5) coordinate (d1);
\path (4,4) coordinate (e1);
\path (5,3.25) coordinate (f1);
\path (6,3.25) coordinate (g1);
\path (7,2.5) coordinate (h1);
\path (7,4) coordinate (i1);

\path (8,3.25) coordinate (l1);
\path (9,3.25) coordinate (m1);
\path (10,2.5) coordinate (n1);
\path (10,4) coordinate (o1);
\path (4,5) coordinate (d2);
\path (4,6.5) coordinate (e2);
\path (5,5.75) coordinate (f2);
\path (6,5.75) coordinate (g2);
\path (7,5) coordinate (h2);
\path (7,6.5) coordinate (i2);

\path (8,5.75) coordinate (l2);
\path (9,5.75) coordinate (m2);
\path (10,5) coordinate (n2);
\path (10,6.5) coordinate (o2);
\draw (d)--(e)--(f)--(d);
\draw (f)--(g);
\draw (i)--(g)--(h)--(i)--(l)--(h);
\draw (n)--(m)--(o)--(n);

\draw (l)--(m);
\draw (d1)--(e1)--(f1)--(d1);
\draw (f1)--(g1);
\draw (i1)--(g1)--(h1)--(i1)--(l1)--(h1);
\draw (n1)--(m1)--(o1)--(n1);

\draw (l1)--(m1);
\draw (d2)--(e2)--(f2)--(d2);
\draw (f2)--(g2);
\draw (i2)--(g2)--(h2)--(i2)--(l2)--(h2);
\draw (n2)--(m2)--(o2)--(n2);

\draw (l2)--(m2);
\draw (e)--(d1);
\draw (e1)--(d2);
\draw (o)--(n1);
\draw (o1)--(n2);
\draw (d) to[out=180,in=180, distance=1.5cm] (e2);
\draw (n) to[out=0,in=0, distance=1.5cm] (o2);
\draw (d) [fill=white] circle (\vr);
\draw (e) [fill=black] circle (\vr);
\draw (f) [fill=black] circle (\vr);
\draw (g) [fill=white] circle (\vr);
\draw (h) [fill=black] circle (\vr);
\draw (i) [fill=black] circle (\vr);
\draw (l) [fill=white] circle (\vr);
\draw (m) [fill=black] circle (\vr);
\draw (n) [fill=white] circle (\vr);
\draw (o) [fill=black] circle (\vr);
\draw (d1) [fill=white] circle (\vr);
\draw (e1) [fill=black] circle (\vr);
\draw (f1) [fill=black] circle (\vr);
\draw (g1) [fill=white] circle (\vr);
\draw (h1) [fill=black] circle (\vr);
\draw (i1) [fill=black] circle (\vr);
\draw (l1) [fill=white] circle (\vr);
\draw (m1) [fill=black] circle (\vr);
\draw (n1) [fill=white] circle (\vr);
\draw (o1) [fill=black] circle (\vr);
\draw (d2) [fill=white] circle (\vr);
\draw (e2) [fill=black] circle (\vr);
\draw (f2) [fill=black] circle (\vr);
\draw (g2) [fill=white] circle (\vr);
\draw (h2) [fill=black] circle (\vr);
\draw (i2) [fill=black] circle (\vr);
\draw (l2) [fill=white] circle (\vr);
\draw (m2) [fill=black] circle (\vr);
\draw (n2) [fill=white] circle (\vr);
\draw (o2) [fill=black] circle (\vr);
\end{tikzpicture}
\end{center}
\vskip -0.35cm
\caption{A $\beta$-set in a triangle-diamond-necklace $H_6$} \label{fig:TDneck1}
\end{figure}

Note that $m(G,3)$ in (claw-free) cubic graphs coincides with $\sigma_{(3,q)}(G)$ for all $q\ge 3$, establishing the $(3,3)$-entry of Table~\ref{Tabela1}. Hence, by Proposition~\ref{prop:relate}, we have $\sigma_{(3,q)}(G) = m(G,3) = \beta(G) = n(G) - \alpha(G)$ for all $q \ge 3$. The result can be extended to $\sigma_{(3,2)}(G)$, which yields the following result.

\begin{proposition}
\label{prop:sigma32}
If $G$ is a connected, claw-free cubic graph of order~$n$, then $\sigma_{(3,q)}(G) = \beta(G)$ for every $q\ge 2$.
\end{proposition}
\proof We have already established that $\sigma_{(3,q)}(G) = \beta(G)$ for every $q\ge 3$, so it remains to resolve the case $q=2$. Consider a $3$-percolating set $P$ of a connected, claw-free cubic graph $G$, and color all its vertices blue. By Proposition~\ref{prop:relate}, $P$ is a vertex cover and so $V(G)\setminus P$ is an independent set in $G$. Now, since an independent set in $G$ contains at most one vertex from each triangle, a vertex cover contains at least two vertices in each triangle. This implies that every vertex in $P$ has a neighbor in $P$. Therefore, 
since $\deg_G(x)=3$ for all $x\in V(G)$, we infer that every (blue) vertex in $P$ has at most two (white) neighbors in $V(G)\setminus P$. Thus, $P$ is a $(3,2)$-spreading set. Since this is true for every $3$-percolating set $P$ of $G$, we infer that $\sigma_{(3,2)}(G) \le m(G,3)=\beta(G)$. 

Conversely, by Observation~\ref{obs:basic},
$\sigma_{(3,2)}(G)\ge \sigma_{(3,3)}(G)$. Combining the obtained inequalities, we infer $$\beta(G)\ge \sigma_{(3,2)}(G)\ge \sigma_{(3,3)}(G)=m(G,3)=\beta(G).~~~\textrm{\QED}$$

\medskip
We remark that $\sigma_{(3,1)}(G)>\beta(G)$ if $G \in \cH_\cub$ is a triangle-diamond necklace. Indeed, in any vertex cover $P$ of $G$, every vertex in $P$ has a neighbor in $P$, and so it does not have at most one neighbor in $V(G)\setminus P$.

Recall that two units in the $\Delta$-D-partition are adjacent if there exists at least one edge joining a vertex in one unit to a vertex in the other unit. We say that a unit in the $\Delta$-D-partition of $G$ is \emph{infected} if all vertices in the unit are infected.

\begin{lemma}
\label{lem:independent set-1}
If $G$ is a connected, claw-free cubic graph that contains only triangle-units, then there exists an independent set in $G$ that contains a vertex from every triangle of~$G$.
\end{lemma}
\proof 
By supposition, every unit in the $\Delta$-D-partition of the connected, claw-free cubic $G$ is a triangle-unit. Since $G$ is diamond-free, we note that the triangle-units in $G$ correspond to the triangles in $G$. The graph $G$ has order~$n = 3t$ where $t$ denotes the number of triangle-units in $G$. By Theorem~\ref{thm:bounds-alpha-cubic}(a), we have $t = \frac{1}{3}n  \le \alpha(G)$. Since every independent set in $G$ contains at most one vertex from every triangle-unit in $G$, we infer that $\alpha(G) \le t = \frac{1}{3}n$. Consequently, $\alpha(G) = \frac{1}{3}n = t$ and every maximum independent set in $G$ contains a vertex from every triangle of $G$.~\QED

\begin{lemma}
\label{lem:independent set-2}
If $G$ is a connected, claw-free cubic graph, then there exists an independent set in $G$ that contains a vertex from every triangle of $G$.
\end{lemma}
\proof 
If the connected, claw-free cubic $G$ that contains only triangle-units, then the result follows from Lemma~\ref{lem:independent set-1}. Hence we may assume that $G$ contains at least one diamond-unit. We now construct an independent set $I$ of $G$ as follows. Initially, we set $I = \emptyset$. For each diamond-unit $D$ of $G$, we add to the set $I$ exactly one of the two vertices of degree~$3$ in the diamond-unit $D$. Thereafter, we delete all diamond-units from the graph $G$. If $G$ contains only diamond-units, then the resulting set $I$ has the desired property that it contains a vertex from every triangle of $G$. Hence we may assume that $G$ contains at least one triangle-unit.

Let $G'$ be the graph obtained by deleting all diamond-units from $G$. We note that every component, $C$, of $G_1$ has the following properties: (1) the component $C$ contains no diamond, (2) every vertex in $C$ belongs to a triangle in the component $C$ and (3) the component $C$ has minimum degree~$2$ (and maximum degree at most~$3$). We now select an arbitrary vertex $u_1$ of degree~$2$ in $C$, add the vertex~$u_1$ to the set $I$, and delete the triangle that contains the vertex~$u_1$. In the resulting graph $G_2$, once again properties (1), (2) and (3) hold, and we select an arbitrary vertex $u_2$ of degree~$2$ in $G_2$, add the vertex~$u_2$ to the set $I$, and delete from $G_2$ the triangle that contains the vertex~$u_2$. Upon completion of this process, the resulting set $I$ is an independent set in $G$ that contains a vertex from every triangle of $G$.~\QED

We note that the statement in Lemma~\ref{lem:independent set-2} is equivalent to the following statement.

\begin{lemma}
\label{lem:independent set}
If $G$ is a connected, claw-free cubic graph, then there exists a vertex cover $P$ such that every triangle of $G$ contains exactly two vertices from $P$. 
\end{lemma}

\begin{proposition}
\label{prp:s31}
If $G$ is a connected, claw-free cubic graph, then $\sigma_{(3,1)}(G) \le \beta(G) + 1$, and this bound is sharp.
\end{proposition}
\proof
Let $G$ be a connected, claw-free cubic graph, and let $P$ be a $3$-percolating set of $G$ satisfying $|P| = m(G,3) = \beta(G)$. We note that $P$ is a vertex cover of $G$ and $V(G) \setminus P$ is a maximum independent set of $G$. Clearly, every vertex, which is not in $P$, has three neighbors in $P$, and so the first condition of the $(3,1)$-spreading rule is always satisfied with respect to $P$. In addition, due to Lemma~\ref{lem:independent set} we may assume that $P$ contains exactly two vertices from every triangle in $G$, or, equivalently, so that the independent set $V(G) \setminus P$ contains a vertex from every triangle in $G$. Now, consider an arbitrary triangle $T$ in $G$, and let $P^* = P \cup \{v\}$, where $v$ is the unique vertex in $V(T) \setminus P$. Let the vertices of $P^*$ be initially infected, and note that (with respect to $P^*$) the triangle $T$ is completely infected. 

Let $u$ be a vertex that is not yet infected and is adjacent to an infected triangle $T'$ (with all three vertices of $T'$ infected). In particular, we note that $u \notin P^*$. Suppose firstly that $u$ is adjacent to two vertices of $T'$. Thus, the vertices in $V(T') \cup \{u\}$ induce a diamond-unit. In this case, the vertex $u$ immediately becomes infected, since both of its neighbors in $T'$ have no other uninfected neighbor. As a result, the (unique) triangle containing the vertex~$u$ becomes an infected triangle, and so the diamond-unit containing $u$ is infected. Suppose next that $u$ is adjacent to exactly one vertex $w'$ of $T'$. Once again in this case, the vertex $u$ immediately becomes infected, since its neighbor $w'$ in $T'$ has no other uninfected neighbor. 
On the other hand, suppose that $u$ is infected and is adjacent to an infected triangle $T'$ (with all three vertices of $T'$ infected), and there exists an uninfected neighbor $v$ of $u$ that belongs to the same unit as $u$. Then $v$ becomes infected, because $u$ has only one uninfected neighbor, namely $v$. 

In all of the above cases, the unit, which is adjacent to an infected unit, becomes infected. 
Since $G$ is connected, and initially there is a triangle infected (making the unit in which the triangle lies also infected), we infer that due to the above arguments all vertices in $G$ become infected.
Thus, $\sigma_{(3,1)}(G) \le |P^*| = |P| + 1 = \beta(G) + 1$, thereby proving the desired upper bound. The case of a triangle-diamond necklace $G \in \cH_\cub$ shows that the bound is sharp.~\QED

\section{$2$-Percolation in claw-free, cubic graphs}
\label{sec:2-percolation}

In this section, we study $2$-percolation in claw-free, cubic graphs. It is easy to see that $m(K_4,2) = 2$. In what follows we therefore restrict our attention to $2$-percolation in connected, claw-free, cubic graphs $G$ different from~$K_4$. Thus, by Lemma~\ref{lem:known}, such a graph $G$ has a $\Delta$-D-partition (in which every unit is a triangle-unit or a diamond-unit). Recall that $u(G)$ is defined as the number of units in this (unique) $\Delta$-D-partition. In what follows, if $D$ is a diamond-unit in the $\Delta$-D-partition of $G$, then a \emph{dominating vertex} in $D$ is a vertex that is adjacent to the three other vertices in the diamond-unit $D$.

\begin{proposition}
\label{prop:2percolation-diamond-neck}
If $G \in \cN_\cub$, then $m(G,2) = u(G)+1$.
\end{proposition}
\proof
Suppose that $G \in \cN_\cub$, and so $G$ is a diamond-necklace $N_k$ for some $k \ge 2$. Thus, $G$ has $u(G) = k$ units, each of which is a diamond-unit. Let $S$ be a $2$-percolating set of $G$ of minimum cardinality, and so $S$ is an $2$-percolating set of $G$ satisfying $|S| = m(G,2)$. By Lemma~\ref{lem:subgraph condition}, we infer that the set $S$ contains at least one vertex from every diamond-unit of $G$. Moreover if $D$ is a diamond-unit of $G$ and the set $S$ contains exactly one vertex from~$D$, then by Lemma~\ref{lem:subgraph condition} we infer that such a vertex of $S$ is a dominating vertex of $D$. In particular, since $S$ contains at least one vertex from every unit in the $\Delta$-D-partition of $G$, we note that $m(G,2) = |S| \ge u(G) = k$.

Suppose that $|S| = k$. By our earlier observations, the set $S$ contains exactly one vertex from every diamond-unit, namely a vertex from each diamond-unit that dominates that unit. The resulting set $S$ is a $2$-packing in $G$; that is, $d_G(u,v) \ge 3$ for every two distinct vertices $u$ and $v$ that belong to~$S$. Moreover in this case, letting $H = G - S$, we note that $H$ is a cycle $C_{3k}$ and every vertex in $H$ has exactly one neighbor that belongs to $V(G) \setminus V(H) = S$. Thus by Lemma~\ref{lem:subgraph condition}, we infer that every $2$-percolating set of $G$ contains at least one vertex of $H$. However this contradicts our supposition that $S$ is a $2$-percolating set of $G$ that contains no vertex of $H$. Hence, $m(G,2) = |S| \ge k+1$.

To establish an upper bound on $m(G,2)$ in this case when $G = N_k$, if $S^*$ consists of a dominating vertex from $k-1$ diamond-units and two dominating vertices from the remaining diamond-unit, then $S^*$ is a $2$-percolating set of $G$ and $|S^*| = k + 1$. For example, if $G = N_4$ (here $k = 4$), then such a set $S^*$ illustrated in Figure~\ref{fig:N4b} by the shaded vertices is a $2$-percolating set of $G$ and $|S^*| = 5$. This implies that $m(G,2) \le |S^*| = k + 1$. Consequently, $m(G,2) = k + 1$. Thus in this case when $G \in \cN_\cub$, we have shown that $m(G,2) = u(G) + 1$.~\QED

\begin{figure}[htb]
\begin{center}
\begin{tikzpicture}[scale=.85,style=thick,x=.85cm,y=.85cm]
\def\vr{2.5pt} 

%
\path (2.25,0.5) coordinate (a);
\path (2.25,2) coordinate (b);
\path (1.5,1.25) coordinate (c);
\path (3,1.25) coordinate (d);
\path (4.5,0.5) coordinate (w);
\path (4.5,2) coordinate (x);
\path (3.75,1.25) coordinate (y);
\path (5.25,1.25) coordinate (z);
\path (6.75,0.5) coordinate (k);
\path (6.75,2) coordinate (l);
\path (6,1.25) coordinate (m);
\path (7.5,1.25) coordinate (n);
\path (8.25,1.25) coordinate (h);
\path (9,0.5) coordinate (i);
\path (9,2) coordinate (j);
\path (9.75,1.25) coordinate (o);
\draw (a) -- (c);
\draw (a) -- (d);
\draw (b) -- (c);
\draw (b) -- (d);
\draw (a) -- (b);
%
\draw (d) -- (y);
\draw (y)--(w)--(z)--(w)--(x)--(y);
\draw (x)--(z);
\draw (h) -- (i);
\draw (h) -- (j);
\draw (i) -- (j);
\draw (k) -- (m);
\draw (k) -- (n);
\draw (l) -- (m);
\draw (l) -- (n);
\draw (k) -- (l);
\draw (m) -- (z);
\draw (n) -- (h);
\draw (i)--(o)--(j);
\draw (c) to[out=110,in=70, distance=2.5cm] (o);
\draw (a) [fill=black] circle (\vr);
\draw (b) [fill=black] circle (\vr);
\draw (c) [fill=white] circle (\vr);
\draw (d) [fill=white] circle (\vr);
\draw (h) [fill=white] circle (\vr);
\draw (i) [fill=white] circle (\vr);
\draw (j) [fill=black] circle (\vr);
\draw (k) [fill=white] circle (\vr);
\draw (l) [fill=black] circle (\vr);
\draw (m) [fill=white] circle (\vr);
\draw (n) [fill=white] circle (\vr);
\draw (w) [fill=white] circle (\vr);
\draw (x) [fill=black] circle (\vr);
\draw (y) [fill=white] circle (\vr);
\draw (z) [fill=white] circle (\vr);
\draw (o) [fill=white] circle (\vr);
%
\end{tikzpicture}
\end{center}
\vskip -0.35cm
\caption{A diamond-necklace $N_4$ and a $2$-percolating set of size~$5$} \label{fig:N4b}
\end{figure}
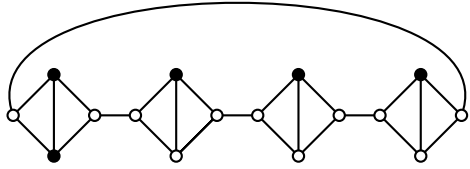

\medskip
 If two distinct units are joined by at least two edges, then we say that these two units are \emph{double-bonded}.

\begin{lemma}\label{lem:special triangle unit}
If $G \ne K_4$ is a connected, claw-free, cubic graph and $G\notin \cN_\cub$, then there exists a triangle-unit $T'$, such that the graph $G-T'$ has at most two components.
\end{lemma}
\proof  Let $\mathcal{T}$ be the set of all triangle-units in $G$. Suppose to the contrary that for every $T\in \mathcal{T}$ the graph $G-T$ consists of three components. Let $T_1 \in \mathcal{T}$ be a triangle-unit such that at least one of the three components of the graph $G-T_1$ does not contain a triangle-unit. Denote this component by $G_1$ and let $D_1$ be the diamond-unit in $G_1$ adjacent to $T_1$. Then in $G_1$, $D_1$ is adjacent to exactly one diamond-unit $D_2$, which is in turn adjacent to another diamond-unit $D_3$ and so on, which yields a contradiction since $G$ is finite. Therefore $G_1$ contains a triangle-unit $T' \notin \mathcal{T}$, which is the final contradiction.~\QED

The following observation will be necessary for proving the subsequent theorem.

\begin{observation}\label{obs:adjacent units}
    Let $G \ne K_4$ be a connected, claw-free, cubic graph, and $U, V$ adjacent units in the $\Delta$-D-partition of $G$. Also, let $uv$ be an edge connecting units $U$ and $V$, where $u \in U$ and $v \in V$. If $u$ is infected and another vertex $v_1 \in V$, where $d(v_1,u)=2$, is also infected, then the whole unit $V$ becomes infected.
\end{observation}

\begin{proof}
Vertex $v$ has two infected neighbors, namely $u$ and $v_1$. After that, if $V$ is a triangle unit, the remaining vertex of $V$ is also infected, and if $V$ is a diamond unit, the remaining two vertices of $V$ become infected after two steps. \QED
\end{proof}

\begin{theorem}
\label{thm:2percolation}
If $G \ne K_4$ is a connected, claw-free, cubic graph of order~$n$ that contains $u(G)$ units, then
\[
m(G,2) = \left\{
\begin{array}{ll}
u(G)+1,  & \mbox{if $G \in \cN_\cub$};  \1 \\
u(G),  & \mbox{otherwise}.
\end{array}
\right.
\]
\end{theorem}
\proof
By Lemma~\ref{lem:subgraph condition}, we infer that every $2$-percolating set of $G$ contains at least one vertex from every triangle-unit and every diamond-unit of $G$. Moreover, if $D$ is a diamond-unit of $G$ and a $2$-percolating set contains exactly one vertex from~$D$, then by Lemma~\ref{lem:subgraph condition} we infer that such a vertex is a dominating vertex of $D$. In particular, since every $2$-percolating set of $G$ contains at least one vertex from every unit in the $\Delta$-D-partition of $G$, we note that $m(G,2) \ge u(G)$. If $G \in \cN_\cub$, then by Proposition~\ref{prop:2percolation-diamond-neck}, $m(G,2) = u(G)+1$. Hence, we may assume that $G \notin \cN_\cub$, for otherwise the desired result follows.

Since every triangle-unit contributes~$3$ to the order of the graph and every diamond-unit contributes~$4$ to the order of the graph, we observe that if $G$ has order~$n$ with $u_t$ triangle-units and $u_d$ diamond-units, then $u(G) = u_t + u_d$ and $n = 3u_t + 4u_d$. By assumption, $u_t \ge 1$. Thus, since $n$ is even, $u_t \ge 2$, that is, $G$ contains at least two triangle-units. According to Lemma~\ref{lem:special triangle unit}, there exists a triangle-unit $T_1$ of $G$, where $V(T_1) = \{t_1,t_2,t_3\}$, such that $G-T_1$ consists of at most two components. This implies that at most one of the edges incident with exactly one vertex of $T_1$ is a bridge. Since $G$ is connected, the triangle-unit $T_1$ is adjacent to at least one other unit. Let $U_1$ be a unit that is adjacent to~$T_1$ such that the edge between $U_1$ and $T_1$ is not a bridge, or $T_1$ and $U_1$ are double-bonded. Finally denote as $G_1$ the component of $G-T_1$ containing $U_1$.

As observed earlier, $m(G,2) \ge u(G)$.  Hence it suffices for us to show that $m(G,2) \le u(G)$.  For this purpose, we construct a $2$-percolating set $S$ of $G$ that contains exactly one vertex from each unit of $G$. Initially, we let $S = \emptyset$. We consider three cases. First, suppose that the units $T_1$ and $U_1$ are not double-bonded.

\smallskip
\emph{Case~1. The unit $U_1$ is a triangle-unit.} Let $V(U_1) = \{a_1,b_1,c_1\}$ and where $t_1c_1$ is an edge. In this case, we add to $S$ the vertices $t_1$ and $a_1$, and so $S = \{t_1,a_1\}$. The vertices in $S$ are indicated by the shaded vertices in Figure~\ref{fig:adj-units}(a). Due to Observation~\ref{obs:adjacent units}, every vertex in the triangle-unit $U_1$ becomes infected.

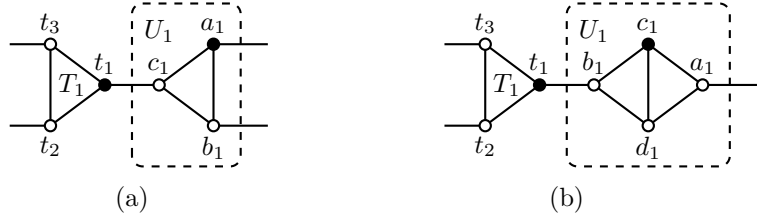
\begin{figure}[htb]
\begin{center}
\begin{tikzpicture}[scale=.85,style=thick,x=.85cm,y=.85cm]
\def\vr{2.5pt} 
\path (0.25,0) coordinate (b);
\path (0.25,1.5) coordinate (c);
\path (1,0) coordinate (d);
\path (1,1.5) coordinate (e);
\path (2,0.75) coordinate (f);
\path (3,0.75) coordinate (g);
\path (4,0) coordinate (h);
\path (4,1.5) coordinate (i);
\path (5,0) coordinate (j);
\path (5,1.5) coordinate (k);
\draw (b)--(d)--(e)--(f)--(d);
\draw (c)--(e);
\draw (f)--(g);
\draw (i)--(g)--(h)--(i);
\draw (h)--(j);
\draw (i)--(k);
\draw (d) [fill=white] circle (\vr);
\draw (e) [fill=white] circle (\vr);
\draw (f) [fill=black] circle (\vr);
\draw (g) [fill=white] circle (\vr);
\draw (h) [fill=white] circle (\vr);
\draw (i) [fill=black] circle (\vr);
\draw[anchor = north] (d) node {{\small $t_2$}};
\draw[anchor = south] (e) node {{\small $t_3$}};
\draw[anchor = south] (f) node {{\small $t_1$}};
\draw[anchor = north] (h) node {{\small $b_1$}};
\draw[anchor = south] (i) node {{\small $a_1$}};
\draw[anchor = south] (g) node {{\small $c_1$}};
\draw (1.4,0.75) node {{\small $T_1$}};
\draw (2.5,-1.35) node {{\small (a)}};
\draw [style=dashed,rounded corners] (2.5,-0.75) rectangle (4.5,2.25);
\draw (3,1.75) node {{\small $U_1$}};

\begin{scope}[shift={(-1,0)}]

    \path (7.25,0) coordinate (b);
    \path (7.25,1.5) coordinate (c);
    \path (8,0) coordinate (d);
    \path (8,1.5) coordinate (e);
    \path (9,0.75) coordinate (f);
    \path (10,0.75) coordinate (g);
    \path (11,0) coordinate (h);
    \path (11,1.5) coordinate (i);
    \path (12,0.75) coordinate (j);
    \path (13,0.75) coordinate (k);
    \draw (b)--(d)--(e)--(f)--(d);
    \draw (c)--(e);
    \draw (f)--(g);
    \draw (i)--(g)--(h)--(i);
    \draw (h)--(j)--(i);
    \draw (j)--(k);
    \draw (d) [fill=white] circle (\vr);
    \draw (e) [fill=white] circle (\vr);
    \draw (f) [fill=black] circle (\vr);
    \draw (g) [fill=white] circle (\vr);
    \draw (h) [fill=white] circle (\vr);
    \draw (i) [fill=black] circle (\vr);
    \draw (j) [fill=white] circle (\vr);
    \draw[anchor = north] (d) node {{\small $t_2$}};
    \draw[anchor = south] (e) node {{\small $t_3$}};
    \draw[anchor = south] (f) node {{\small $t_1$}};
    \draw[anchor = north] (h) node {{\small $d_1$}};
    \draw[anchor = south] (i) node {{\small $c_1$}};
    \draw[anchor = south] (g) node {{\small $b_1$}};
    \draw[anchor = south] (j) node {{\small $a_1$}};
    \draw (8.4,0.75) node {{\small $T_1$}};
    \draw (9.5,-1.35) node {{\small (b) }};
    \draw [style=dashed,rounded corners] (9.5,-0.75) rectangle (12.5,2.25);
    \draw (10,1.75) node {{\small $U_1$}};    
\end{scope}
\path (13,0.75) coordinate (x);
\path (14,0.75) coordinate (g);
\path (15,0) coordinate (h);
\path (15,1.5) coordinate (i);
\path (16,0) coordinate (j);
\path (16,1.5) coordinate (k);
\path (17,0.75) coordinate (l);
\path (18,0.75) coordinate (m);

\draw (i)--(g)--(h)--(i);
\draw (h)--(j);
\draw (j)--(k)--(i);
\draw (x)--(g);
\draw (j)--(l)--(k);
\draw (l)--(m);

\draw (g) [fill=white] circle (\vr);
\draw (h) [fill=white] circle (\vr);
\draw (i) [fill=black] circle (\vr);
\draw (j) [fill=black] circle (\vr);
\draw (k) [fill=white] circle (\vr);
\draw (l) [fill=white] circle (\vr);

\draw[anchor = north] (h) node {{\small $t_2$}};
\draw[anchor = south] (i) node {{\small $t_1$}};
\draw[anchor = south] (g) node {{\small $t_3$}};
\draw[anchor = north] (j) node {{\small $u$}};

\draw (14.6,0.75) node {{\small $T_1$}};
\draw (15.5,-1.35) node {{\small (c) }};
\draw [style=dashed,rounded corners] (15.5,-0.75) rectangle (17.5,2.25);
\draw (17,1.75) node {{\small $U_1$}};

\end{tikzpicture}
\end{center}
\vskip -0.35cm
\caption{Possible adjacent units in the proof of Theorem~\ref{thm:2percolation}} \label{fig:adj-units}
\end{figure}

\smallskip
\emph{Case~2. The unit $U_1$ is a diamond-unit.} Let $V(U_1) = \{a_1,b_1,c_1,d_1\}$ and where $a_1b_1$ is the missing edge in $U_1$ and where $t_1b_1$ is an edge of $G$. In this case, we add to $S$ the vertices $t_1$ and $c_1$, and so $S = \{t_1,c_1\}$. The vertices in $S$ are indicated by the shaded vertices in Figure~\ref{fig:adj-units}(b). Due to Observation~\ref{obs:adjacent units}, every vertex in the diamond-unit $U_1$ becomes infected.

\emph{Case~3}. The units $T_1$ and $U_1$ are double-bonded. Thus there is a vertex $u \in U_1$ that is joined to a vertex of $T_1$ different from~$t_1$. Renaming the vertices $t_2$ and $t_3$ if necessary, we may assume that $ut_2$ is such an edge between the units $T_1$ and $U_1$. In particular, we note that $t_2 \notin S$ and let $S=\{t_1,u\}$. The vertices in $S$ are indicated by the shaded vertices in Figure~\ref{fig:adj-units}(c). Again, due to Observation~\ref{obs:adjacent units}, every vertex in both units $T_1$ and $U_1$ becomes infected. In this case, the set $S$ consists of two vertices, one vertex from each of the units $U_1$ and $T_1$.

Cases~1,~2 and~3 occur in Step~1 of our procedure to construct the set $S$. Thus after Step~1, the initial set $S$ consists of two vertices, namely a vertex~$t_1$ from the triangle-unit~$T_1$ and one vertex from the unit $U_1$ that is adjacent to $T_1$. Moreover in all three cases, every vertex in the unit $U_1$ becomes infected. We now proceed to Step~2 of our procedure to construct the set $S$.

Let $U_2$ be a unit different from $T_1$ that is adjacent to the infected unit~$U_1$, and let $uv$ be an edge that joins a vertex $u \in U_1$ and a vertex $v \in U_2$. Note that if there is no such unit $U_2$, then $G$ consist only of units $T_1$ and $U_1$ and is now fully infected, yielding the desired result. Since $U_2 \ne T_1$, we note that $v \ne t_1$. In particular, we note that $v \notin S$ since at this stage of the construction the set $S$ only contains the vertex~$t_1$ and one other vertex, namely a vertex from~$U_1$.

Suppose that $U_2$ is a triangle-unit. Let $V(U_2) = \{v,v_1,v_2\}$. We now add to the set $S$ exactly one of the vertices $v_1$ and $v_2$. By symmetry, we may assume that $v_1$ is added to the set $S$. Once again, due to Observation~\ref{obs:adjacent units}, every vertex in the  triangle-unit $U_2$ becomes infected.

Suppose next that $U_2$ is a diamond-unit. Let $V(U_2) = \{v,v_1,v_2,v_3\}$, where $v_1$ and $v_2$ are the dominating vertices of $U_2$ (and so, $vv_3$ is the missing edge in $D_v$). We now add to the set $S$ exactly one of the dominating vertices of $U_2$. By symmetry, we may assume that dominating vertex $v_1$ of $U_2$ is added to~$S$. We again invoke Observation~\ref{obs:adjacent units}, thus every vertex in the diamond-unit $U_2$ becomes infected.

After Step~2 of our procedure to construct the set $S$, every vertex in the unit $U_2$ becomes infected, and the set $S$ contains exactly one vertex from each of the units $T_1, U_1$ and $U_2$. 

Suppose that $i \ge 2$ and after Step~$i$ of our procedure to construct the set $S$, the units $U_1,U_2, \ldots,U_i$ are all infected, and the set $S$ contains exactly one vertex from each of the units $U_1,U_2, \ldots,U_i$. Moreover if $T_1 = U_j$ for some $j \in [i] \setminus \{1\}$, then after Step~$i$ we have $|S| = i$ and the set $S$ contains exactly one vertex from each of the units $U_1,U_2, \ldots,U_i$, while if $T_1 \ne U_j$ for any $j \in [i]$, then $|S| = i+1$ and $S$ contains one vertex from each of the units $U_1,U_2, \ldots,U_i,T_1$.

In Step~$i+1$, let $U_{i+1}$ be a unit different the units $U_1,U_2, \ldots,U_i$ that is adjacent to an infected unit~$U_j$ for some $j \in [i]$, and let $uv$ be an edge that joins a vertex $u \in U_j$ and a vertex $v \in U_{i+1}$.

Suppose that $U_{i+1}$ is a triangle-unit. Let $V(U_{i+1}) = \{v,v_1,v_2\}$. In this case, either $U_{i+1} = T_1$ and renaming vertices of $T_1$ if necessary we may assume that $v_1 = t_1$, or $U_{i+1} \neq T_1$ and we add to the set $S$ the vertex $v_1$. Suppose next that $U_{i+1}$ is a diamond-unit. Let $V(U_{i+1}) = \{v,v_1,v_2,v_3\}$, where $v_1$ and $v_2$ are the dominating vertices of $U_{i+1}$ (and so, $vv_3$ is the missing edge in $D_v$). We now add to the set $S$ the dominating vertex $v_1$ of $U_{i+1}$. In either of these cases the vertices $u$ and $v_1$ satisfy the conditions of Observation~\ref{obs:adjacent units}, therefore the whole unit $U_{i+1}$ becomes infected.

Continuing in the described way, the above process continues until every unit of $G_1$ and also the unit $T_1$ become infected. If $G-T_1$ is connected, then all vertices of $G$ have become infected and $S$ contains exactly one vertex from each unit of $G$, thus the proof is complete. Otherwise, if $G-T_1$ is not connected, then it has at most two components, and let $G_2$ be the component of $G-T_1$ different from $G_1$. Let $U_1'$ be the unit in $G_2$ that is adjacent to $T_1$, and let $v_2a$ be the bridge of $G$ connecting $T_1$ with $U_1'$. Now, adding a dominating vertex $b$ of $U_1'$, different from $a$, we infer that all vertices of $U_1'$ become infected (using the analogous arguments as in the previous cases). Now, we can continue with the same process, continuing by infecting the units $U_i'$ in $G_2$ that are adjacent to an already infected unit by adding to $S$ appropriately selected dominating vertex of $U_i'$. Since $G_2$ is connected, we infer that all vertices of $G_2$ become infected, where we put in $S$ exactly one vertex of each unit. Thus, $S$ is a $2$-percolating set of $G$, implying that $m(G,2) \le u(G)$. As observed earlier, $m(G,2) \ge u(G)$. Consequently, $m(G,2) = u(G)$. This completes the proof of Theorem~\ref{thm:2percolation}.~\QED

\medskip
By definition of a $(p,q)$-spreading set in a graph, if $G$ is cubic graph, then a set is a $(2,3)$-spreading set if and only if it is $2$-percolating set, implying that $\sigma_{(2,3)}(G) = m(G,2)$. As an immediate consequence of Theorem~\ref{thm:2percolation}, we infer the following result.

\begin{corollary}
\label{cor:2percolation}
If $G \ne K_4$ is a connected, claw-free, cubic graph of order~$n$ that contains $u(G)$ units, then
\[
\sigma_{(2,3)}(G) = \left\{
\begin{array}{ll}
u(G)+1,  & \mbox{if $G \in \cN_\cub$};  \1 \\
u(G),  & \mbox{otherwise}.
\end{array}
\right.
\]
\end{corollary}

We are now in a position to establish the following relationships between the $(2,2)$-spreading number and the $(2,3)$-spreading number of a connected, claw-free, cubic graph.

\begin{proposition}
\label{prop1:22-spread-23spread}
If $G = K_4$ or if $G \in \cN_\cub$, then $\sigma_{(2,2)}(G) = \sigma_{(2,3)}(G)$.
\end{proposition}
\proof If $G = K_4$, then $\sigma_{(2,2)}(G) = \sigma_{(2,3)}(G) = 2$. Hence we may assume that $G \ne K_4$.
Suppose that $G \in \cN_\cub$. By Corollary~\ref{cor:2percolation}, we have $\sigma_{(2,3)}(G) = u(G) + 1$. Moreover, as shown in the proof of Proposition~\ref{prop:2percolation-diamond-neck}, if $G = N_k$ for some $k \ge 2$, then we can choose the $(2,3)$-spreading set $S$ of $G$ to consists of a dominating vertex from $k-1$ diamond-units and two dominating vertices from the remaining diamond-unit in $G$ (as illustrated in Figure~\ref{fig:N4b} in the case when $k = 4$). However such a set $S$ is also a $(2,2)$-spreading set $S$ of $G$, and so $\sigma_{(2,2)}(G) \le |S| = \sigma_{(2,3)}(G)$. By Observation~\ref{obs:basic}, $\sigma_{(2,3)}(G) \le \sigma_{(2,2)}(G)$. Consequently, if $G \in \cN_\cub$, then $\sigma_{(2,2)}(G) = \sigma_{(2,3)}(G)$.~\QED

\begin{proposition}
\label{prop2:22-spread-23spread}
If $G$ is a connected, claw-free, cubic graph, then $\sigma_{(2,2)}(G) \le \sigma_{(2,3)}(G) + 1$.
\end{proposition}
\proof
Let $G$ be a connected, claw-free, cubic graph. If $G = K_4$ or if $G \in \cN_\cub$, then by Proposition~\ref{prop1:22-spread-23spread} we have $\sigma_{(2,2)}(G) = \sigma_{(2,3)}(G)$. Hence we may assume that $G \ne K_4$ and $G \notin \cN_\cub$. To prove that $\sigma_{(2,2)}(G) \le \sigma_{(2,3)}(G)+1$, let $S$ be a $(2,3)$-spreading set of $G$ as constructed in the proof of Theorem~\ref{thm:2percolation}. Adopting the notation from the proof of Theorem~\ref{thm:2percolation}, in Step~1 of our procedure to construct the set $S$ we add one additional vertex to the set $S$ as follows. If the unit $U_1$ is a triangle-unit (see Figure~\ref{fig:adj-units}(a)), then we add the common neighbor of $t_1$ and $a_1$, namely the vertex~$c_1$, to the set $S$. If the unit $U_1$ is a diamond-unit (see Figure~\ref{fig:adj-units}(b)), then we add the common neighbor of $t_1$ and $c_1$, namely the vertex~$b_1$, to the set $S$. Let $S^*$ denote the resulting set upon completion of the construction of the set $S$, and so $|S^*| = |S| + 1 = \sigma_{(2,3)}(G) + 1$. With the addition of the vertex $c_1$ in Case~1 or the vertex~$b_1$ in Case~2, the set $S^*$ is a $(2,2)$-spreading set $S$ of $G$, implying that $\sigma_{(2,2)}(G) \le |S^*| = |S| + 1 = \sigma_{(2,3)}(G) + 1$.~\QED

\medskip
Every $(2,3)$-spreading set of $G$ is by definition also a $(2,2)$-spreading set of $G$, and so $\sigma_{(2,3)}(G) \le \sigma_{(2,2)}(G)$. Hence as a consequence of Proposition~\ref{prop2:22-spread-23spread} and applying also Corollary~\ref{cor:2percolation}, we have the following result.

\begin{corollary}
\label{cor:sigma22-23}
If $G$ is a connected, claw-free, cubic graph, then
\[
\sigma_{(2,3)}(G) \le \sigma_{(2,2)}(G) \le \sigma_{(2,3)}(G)+1.
\]
In addition, $u(G) \le \sigma_{(2,2)}(G) \le u(G)+1$.
\end{corollary}

If $G=F_{2k} \in \cT_\cub$ then it is not difficult to see that $S=\{z_i, \, : \, i \in [2k-1] \}\cup\{x_{2k}\}$ is a $(2,2)$-spreading set. Notably, the process starts with the infection of vertex $x_1$, and then continues with $y_1,y_2$ and so on. Therefore $\sigma_{(2,2)}(G)=u(G)$. On the other hand, if $G \in \cN_\cub$, then $\sigma_{(2,2)}(G)=u(G)+1$. In both of these cases it holds that $\sigma_{(2,2)}(G)=\sigma_{(2,3)}(G)$, however this does not hold in general. For instance, see the graph $G$ in Figure~\ref{fig:s22s23} for which we have $\sigma_{(2,3)}=u(G)$, where the shaded vertices indicated in the figure form a $(2,3)$-spreading set of $G$. Yet, one can verify that $\sigma_{(2,2)}=u(G)+1$.

\begin{figure}[htb]
\begin{center}
\begin{tikzpicture}[scale=.85,style=thick,x=.85cm,y=.85cm]
\def\vr{2.5pt} 

%
\path (2.25,0.5) coordinate (a);
\path (2.25,2) coordinate (b);
\path (1.5,1.25) coordinate (c);
\path (3,1.25) coordinate (d);
%

\draw (a) -- (c);
\draw (a) -- (d);
\draw (b) -- (c);
\draw (b) -- (d);
\draw (c) -- (d);

\path (4.25,0.5) coordinate (a1);
\path (4.25,2) coordinate (b1);
\path (3.5,1.25) coordinate (c1);
\path (5,1.25) coordinate (d1);
%

\draw (a1) -- (c1);
\draw (a1) -- (d1);
\draw (b1) -- (c1);
\draw (b1) -- (d1);
\draw (c1) -- (d1);

\path (6.25,0.5) coordinate (a2);
\path (6.25,2) coordinate (b2);
\path (5.5,1.25) coordinate (c2);
\path (7,1.25) coordinate (d2);
%

\draw (a2) -- (c2);
\draw (a2) -- (d2);
\draw (b2) -- (c2);
\draw (b2) -- (d2);
\draw (c2) -- (d2);

\path (3.5,3.75) coordinate (x);
\path (4.25,3) coordinate (y);
\path (5,3.75) coordinate (z);

\draw (x) -- (y)--(z)--(x);

\path (3.5,-1.25) coordinate (x1);
\path (4.25,-0.5) coordinate (y1);
\path (5,-1.25) coordinate (z1);

\draw (x1) -- (y1)--(z1)--(x1);

\draw (x1)--(a);
\draw (y1)--(a1);
\draw (z1)--(a2);

\draw (x)--(b);
\draw (y)--(b1);
\draw (z)--(b2);

\draw (a) [fill=white] circle (\vr);
\draw (b) [fill=white] circle (\vr);
\draw (c) [fill=black] circle (\vr);
\draw (d) [fill=white] circle (\vr);
\draw (a1) [fill=white] circle (\vr);
\draw (b1) [fill=white] circle (\vr);
\draw (c1) [fill=black] circle (\vr);
\draw (d1) [fill=white] circle (\vr);
\draw (a2) [fill=white] circle (\vr);
\draw (b2) [fill=white] circle (\vr);
\draw (c2) [fill=black] circle (\vr);
\draw (d2) [fill=white] circle (\vr);
\draw (x) [fill=white] circle (\vr);
\draw (y) [fill=white] circle (\vr);
\draw (z) [fill=black] circle (\vr);
\draw (x1) [fill=white] circle (\vr);
\draw (y1) [fill=black] circle (\vr);
\draw (z1) [fill=white] circle (\vr);

%
\end{tikzpicture}
\end{center}
\vskip -0.35cm
\caption{A graph $G$ with $\sigma{(2,3)}=u(G)$ and $\sigma{(2,2)}=u(G)+1$} \label{fig:s22s23}
\end{figure}
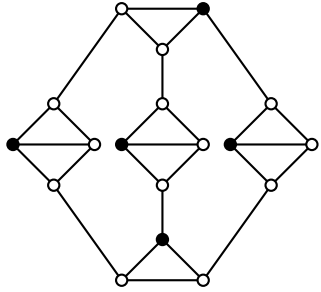

Finally, we consider $(2,1)$-spreading, and bound the corresponding invariant in claw-free cubic graphs. Again we see that it can achieve only two possible values.

\begin{theorem}
\label{thm:s21}
If $G \ne K_4$ is a connected, claw-free, cubic graph of order~$n$ that contains $u(G)$ units, then $u(G) +1 \le \sigma_{(2,1)}(G) \le u(G) + 2$.
\end{theorem}
\proof 
First note that due to Lemma~\ref{lem:subgraph condition}  a $(2,1)$-spreading set $S$ needs to contain at least one vertex in each unit of $G \ne K_4$. Now, if $S$ is a set of vertices that contains exactly one vertex from each unit of $G$, then every vertex $v \in S$ has at least two neighbors in $V(G) \setminus S$, and therefore such a set $S$ cannot be a $(2,1)$-spreading set. Hence, $u(G) +1 \le \sigma_{(2,1)}(G)$.

We show next that $\sigma_{(2,1)}(G) \le u(G) + 2$ by constructing a $(2,1)$-spreading set $S$ of $G$ satisfying $|S| \le u(G) + 2$. Let $T$ be an arbitrary triangle in $G$ (that may also be part of a diamond-unit).  Initially, we let $S = V(T)$, and so the vertices in $T$ form the set of initially infected vertices. We now enlarge the set $S$ as follows. Suppose that $T$ belongs to a diamond-unit $D$. In this case, we let $v$ be the vertex in the diamond-unit $D$ that is not in $T$. The vertex~$v$ immediately becomes infected in the $(2,1)$-spreading process, since both of its neighbors in $T$ have no other uninfected neighbor. As a result, all vertices in the diamond-unit $D$ become infected, that is, $D$ becomes an infected unit.
Suppose next that $T$ is a triangle-unit of $G$. Then $T$ is already an infected unit.

Finally let $T'$ be the infected unit containing $T$ (that is, $T'=D$ or $T'=T$). We now extend the set $S$ in the same manner as in the proof of Theorem~\ref{thm:2percolation}. Let $U$ be a unit adjacent to~$T'$. If $U$ is a triangle-unit, then we add to $S$ a vertex in $U$ that is not adjacent to any vertex in the unit $T'$. If $U$ is a diamond-unit, then we add to $S$ a vertex in $U$ that is a dominating vertex of the unit $U$. Proceeding analogously as in the proof of Theorem~\ref{thm:2percolation}, this results in all vertices of the unit $U$ becoming infected in the $(2,1)$-spreading process.

Continuing in this way by considering a vertex not yet infected that is adjacent to an infected triangle, we obtain a $(2,1)$-spreading set $S$ of $G$ starting with the set $V(T)$ and adding exactly one vertex from every unit of $G$ that does not contain the triangle $T$. Thus,  $\sigma_{(2,1)}(G) \le |S| = u(G) + 2$.~\QED

\section{Concluding remarks}

We end this paper with some remarks concerning the computational complexity of determining $\sigma_{(p,q)}(G)$ in claw-free cubic graphs $G$. We have already mentioned that it is unclear whether one can determine the zero forcing number of a claw-free cubic graph efficiently. For some of the values of $p$ and $q$, where $p>1$ or $q>1$, there exists a polynomial-time algorithm to determine $\sigma_{(p,q)}(G)$.  First, determining the independence number in claw-free graphs can be done in polynomial time (see~\cite{min-1980}, where a polynomial-time algorithm is presented even for the weighted version of the problem). Therefore the problem of determining $\sigma_{(3,q)}(G)$ is polynomial in claw-free cubic graphs $G$ for all $q\ge 2$. Similarly, one can efficiently determine the number of units in a claw-free cubic graph. Thus the problem of determining $\sigma_{(2,q)}(G)$ is also polynomial in claw-free cubic graphs $G$ for any $q\ge 3$. 

The following problems remain open:

\begin{problem}
Provide a structural characterization of the connected claw-free cubic graphs $G$  for which $\sigma_{(2,2)}=u(G)$ (resp.,~$u(G)+1$). Is there a polynomial-time algorithm to recognize the corresponding classes of connected claw-free cubic graphs?
\end{problem}

\begin{problem}
Provide a structural characterization of the connected claw-free cubic graphs $G$  for which $\sigma_{(2,1)}=u(G)+1$ (resp.,~$u(G)+2$). Is there a polynomial-time algorithm to recognize the corresponding classes of connected claw-free cubic graphs?
\end{problem}

\begin{problem}
Provide a structural characterization of the connected claw-free cubic graphs $G$  for which $\sigma_{(3,1)}=\beta(G)$ (resp.,~$\beta(G)+1$). Is there a polynomial-time algorithm to recognize the corresponding classes of connected claw-free cubic graphs?
\end{problem}

\section*{Acknowledgements}

We are grateful to anonymous reviewer for meticulous reading of the first version of the manuscript and for numerous useful suggestions. 
Bo\v{s}tjan Bre\v{s}ar was supported by the Slovenian Research Agency (ARRS) under the grants P1-0297, N1-0285, and J1-4008. Jaka Hed\v zet was supported by the Slovenian Research Agency (ARRS) under the grant P1-0297. Michael A. Henning was supported in part by the South African National Research Foundation (grants 132588, 129265) and the University of Johannesburg.






\begin{thebibliography}{99}


\bibitem{AIM} AIM Minimum Rank-Special Graphs Work Group, {\it Zero forcing sets and the minimum rank of graphs}, Linear algebra Appl.\ {\bf 428} (2008), 1628--1648.

\bibitem{acdp} D.~Amos, Y.~Caro, R.~Davila, R.~Pepper, {\it Upper bounds on the $k$-forcing number of a graph}, Discrete Appl.\ Math.\ {\bf 181} (2015), 1--10.

\bibitem{BaHe-22a} A.~Babikir and M.A.~Henning, {\it Triangles and (total) domination in subcubic graphs}, Graphs Combin. \textbf{38} (2022), Paper No. 28, 17 pp.

\bibitem{BaHe-22b} A.~Babikir and M.A.~Henning, {\it Upper total domination in claw-free cubic graphs}, Graphs Combin. \textbf{38} (2022),  Paper No. 172, 15 pp.



\bibitem{BP-1998} J.\ Balogh, G.\ Pete, {\it Random disease on the square grid}, Random Str.\ Alg.\ \textbf{13} (1998), 409--422.

\bibitem{Bol-2006} B.\ Bollob\'{a}s, {\it The Art of Mathematics: Coffee Time in Memphis}, Cambridge Univ.\ Press, New York, 2006.

\bibitem{bdeh} B.~Bre{\v{s}}ar, T.~Dravec, A.~Erey, J.~Hed\v{z}et, {\it Spreading in graphs}, Discrete Appl.\ Math.\ \textbf{353} (2024), 139--150.


\bibitem{CLR-1979} J.\ Chalupa, P.L.\ Leath, G.R.\ Reich, {\it Bootstrap percolation on a Bethe lattice}, J.\ Physics C: Solid State Physics \textbf{12} (1979), L31.

\bibitem{dahe-2020} R.~Davila, M. A.~Henning, {\it Zero forcing in claw-free cubic graphs}, Bull.\ Malays.\ Math.\ Sci.\ Soc.\ \textbf{43} (2020), 673--688.





\bibitem{Faudree-92} R.J.~Faudree, R.J.~Gould, M.S.~Jacobson, L.M.~Lesniak, T.E.~Lindquester, {\it On independent generalized degrees and independence numbers in $K(1,m)$-free graphs}, Discrete Math. \textbf{103} (1992), 17--24.





\bibitem{HaHeHe-23} T.W.~Haynes, S.T.~Hedetniemi, M.A.~Henning, {\it Domination in Graphs: Core Concepts},  Series: Springer Monographs in Mathematics, Springer, Cham, 2023. 

\bibitem{he2024} M.~He, H.~Li, N.~Song, S.~Ji, 
{\it The zero forcing number of claw-free cubic graphs},
Discrete Appl.\ Math.\ \textbf{359} (2024), 321--330.

\bibitem{HH-2024+} J.~Hed\v{z}et, M.A.~Henning, {\it $3$-neighbor bootstrap percolation on grids}, Discuss.\ Math.\ Graph Theory \textbf{45} (2025), 283--310.


\bibitem{HeLo-12} M.A.~Henning, C.~L\"{o}wenstein, {\it Locating-total domination in claw-free cubic graphs} Discrete Math. {\bf 312} (2012), 3107--3116.



\bibitem{HLS} L.~Hogben, J.C.-H.~Lin, B.L.~Shader, {\it Inverse problems and zero forcing for graphs}, Mathematical Surveys and Monographs 270, American Mathematical Society, Providence, 2022.


\bibitem{min-1980} G.J.~Minty,
{\it On maximal independent sets of vertices in claw-free graphs},
J.\ Combin.\ Theory Ser.\ B {\bf 28} (1980), 284--304.

\bibitem{LiVi-90} H.~Li, C.~Virlouvet, {\it Neighborhood conditions for claw-free Hamiltonian graphs}, Twelfth British Combinatorial Conference (Norwich, 1989), Ars Combin.\ {\bf 29} (1990), 109--116.








\end{thebibliography}
\end{document}